\documentclass[11pt,reqno]{amsart}

\usepackage[a4paper,
            top=30mm,bottom=30mm,
            left=28mm,right=28mm]{geometry}

\usepackage{amsmath, amsthm, amsfonts, amssymb, amscd, mathtools, bm} 
\usepackage{xcolor} 
\usepackage[unicode=true]{hyperref} 
\usepackage[shortlabels]{enumitem} 
\usepackage[T1]{fontenc}
\usepackage{empheq}
\usepackage{placeins}

\usepackage{tikz} 
\usepackage{graphicx}  
\usepackage{subcaption, wrapfig} 

\theoremstyle{plain}
\newtheorem{theorem}{Theorem}[section]
\newtheorem{lemma}[theorem]{Lemma}

\theoremstyle{definition}

\newtheorem{example}[theorem]{Example}

\theoremstyle{remark}

\numberwithin{equation}{section}

\newcommand{\eps}{\varepsilon}

\newcommand{\dx}{\mathrm{d} x}

\newcommand{\N}{\mathbb{N}}
\newcommand{\R}{\mathbb{R}}

\newcommand{\ap}{\alpha_p}

\newcommand{\F}{\mathrm{F}}

\newcommand{\idot}{\!\cdot\!}

\title[]{\boldmath Kolmogorov--Riesz compactness in \\ asymptotic $L_p$ spaces}

\author[]{Nuno J. Alves}
\address[N. J. Alves]{
      University of Vienna, Faculty of Mathematics, Oskar-Morgenstern-Platz 1, 1090 Vienna, Austria.}
\email{nuno.januario.alves@univie.ac.at}

\begin{document}

\begin{abstract}
We extend the classical Kolmogorov--Riesz compactness theorem to the setting of asymptotic $L_p$ spaces on $\R^n$. These are nonlocally convex $F$-spaces that contain the standard $L_p$ spaces as dense subspaces and include all measurable functions supported on sets of finite measure. In contrast with the classical $L_p$ setting, an additional almost equiboundedness condition is needed, and we prove that together with the natural tail and translation conditions it characterizes relative compactness. We conclude with illustrative examples.
\end{abstract}

\maketitle
\thispagestyle{empty}

\section{Introduction}

The classical Kolmogorov--Riesz compactness theorem provides necessary and sufficient conditions for a family of functions in $L^p(\R^n)$, with $1 \leq p < \infty$, to be totally bounded, and hence relatively compact due to completeness. This result plays a key role in establishing existence results for partial differential equations.
\par
We recall the statement of the theorem using the standard norm $\| \idot \|_p$ on $L^p(\R^n)$:

\begin{theorem}[Kolmogorov--Riesz compactness theorem in $L^p(\R^n)$ {\cite{hanche2010kolmogorov, hanche2019improvement}}] \label{thm_KR}
A subset $\mathcal{F} \subseteq L^p(\R^n)$, with $1 \leq p < \infty$, is totally bounded with respect to $\| \idot \|_p$ if and only if the following two conditions hold:
\begin{enumerate}[(i)]
\item For each $\varepsilon > 0$, there exists $R > 0$ such that
\[
\int_{|x| > R} |f|^p \, \dx < \varepsilon^p
\]
for all $f \in \mathcal{F}$.
\smallskip
\item For each $\varepsilon > 0$, there exists $r > 0$ such that
\[
\int_{\R^n} | \tau_y f - f |^p \, \dx < \varepsilon^p
\]
for every $y \in \R^n$ with $|y| < r$, and all $f \in \mathcal{F}$, where $\tau_y f(x) = f(x+y)$.
\end{enumerate}
\end{theorem}

The classical version of this theorem included a third condition: boundedness in $L^p(\R^n)$ of the family $\mathcal{F}$. However, this assumption has been shown to be redundant, as it follows from conditions (i) and (ii); see \cite{hanche2019improvement}. A related improvement in the setting of bounded metric measure spaces can be found in \cite{gorka2017light}. For a historical account and a proof of Theorem~\ref{thm_KR} based on a general compactness lemma in metric spaces, we refer to \cite{hanche2010kolmogorov}.

\par

In this note, we establish an analogous compactness criterion in the nonlocally convex setting of \textit{asymptotic $L_p$ spaces} on $\R^n$; see Theorem~\ref{main_thm}. These spaces, denoted by $\Lambda^p(\R^n)$, were introduced in \cite{alves2025F} and consist of real-valued measurable functions that are \textit{almost in $L_p$}, in the sense that they belong to $L^p(\R^n)$ outside sets of arbitrarily small measure. The topology is given by \textit{asymptotic $L_p$-convergence} (abbreviated as \textit{$\ap$-convergence}), which endows the space with a complete metric structure. More precisely,
\[
\Lambda^p(\R^n) = \Big\{ f:\R^n \to \R \ \text{measurable} \ : \ \forall \delta> 0 \ \exists E_\delta \text{ with } |E_\delta| < \delta \text{ and } f \chi_{E_\delta^c} \in L^p(\R^n) \Big\}
\]
where $\chi_E$ denotes the characteristic function of the set $E$, and $E^c = \R^n \setminus E$. The topology is generated by the $\F$-norm
\[
\|f\|_{\ap} := \|\! \min(|f|,1) \|_p.
\]
This $\F$-norm generates the topology of $\ap$-convergence and makes $\Lambda^p(\R^n)$ into a complete metric space; see \cite{alves2025F} for details. Section~\ref{section_asymptotic_spaces} recalls the main properties of these spaces.

\par

Given the apparent similarity between $\Lambda^p(\R^n)$ and $L^p(\R^n)$, a natural first step toward establishing a compactness criterion for $\Lambda^p(\R^n)$ is to try adapting the proof of Theorem~\ref{thm_KR}. However, the lack of homogeneity of the $\F$-norm prevents a direct adaptation and leads to the need for an additional condition: \textit{almost equiboundedness}, which controls the measure of the regions where functions take large values. We show that this condition is also necessary; see condition (iii) in Theorem~\ref{main_thm} and Lemma~\ref{lemma3}.

\par


\par

To the best of the author’s knowledge, Theorem~\ref{main_thm} appears to be the first Kolmogorov--Riesz type compactness criterion on an unbounded domain for a nonlocally convex $\F$-space whose topology is generated by a nonhomogeneous $\F$-norm. An $\F$-space is a completely metrizable topological vector space with a translation-invariant metric; standard examples include, for $0 < p < 1$, the spaces $L_p$, $\ell_p$, and the Hardy spaces $H^p$ of analytic functions \cite{kalton1984space}.



\par 

Over the past two decades, extensions of Theorem~\ref{thm_KR} have been developed in a wide range of functional settings. These include locally compact abelian groups \cite{georgescu2004riesz}; Banach function spaces, including grand variable exponent Lebesgue spaces \cite{gorka2016arzela, gorka2017corrigendum, gorka2019banach}; variable exponent Lebesgue spaces \cite{rafeiro2009kolmogorov, gorka2015almost, bandaliyev2017compactness, bandaliyev2018relatively}; variable exponent Morrey spaces \cite{bandaliyev2021relatively}; grand Lebesgue spaces \cite{rafeiro2015compactness}; quasi-Banach function spaces \cite{caetano2016compactness, guo2020relatively}; and spaces of variable integrability and summability \cite{dymek2023compactness}.

\par 

The manuscript is organized as follows. In Section~\ref{section_asymptotic_spaces}, we recall the main properties of the asymptotic $L_p$ spaces. Section~\ref{section_main_result} contains the statement of the main result, Theorem~\ref{main_thm}, which is then proved in Sections~\ref{section_proof_necessity} and \ref{section_proof_sufficiency}. Finally, in Section~\ref{section_examples}, we present examples that illustrate our result.

\section{Asymptotic \texorpdfstring{$L_p$}{Lp} spaces} \label{section_asymptotic_spaces}

In this section, we give an overview of the main properties of the asymptotic $L_p$ spaces $\Lambda^p(\R^n)$.  
\par
This line of research was initiated in \cite{alves2024mode} with the introduction of asymptotic $L_p$-convergence, motivated by a question related to convergence in relative entropy. A sequence of measurable functions $\{f_k\}_{k\in \N}$ is said to $\ap$-converge to a function $f$ if there exists a sequence of measurable sets $\{B_k\}_{k\in\N}$ such that  
\[
\int_{B_k} |f_k - f|^p \, \dx \to 0, \quad |B_k^c| \to 0, \quad \text{as } k \to \infty.
\]  
Basic properties of this mode of convergence were studied in \cite{alves2024mode}, and it was shown in \cite{alves2024relation} that, on finite measure spaces, it is equivalent to convergence in measure.  
\par
It is easy to see that $\ap$-convergence is generated by the $\F$-norm $\| \idot \|_{\ap}$ defined above. An $\F$-norm is a functional similar to a norm, except that homogeneity is replaced by the following two conditions:  
\[
\|\lambda f \|_{\ap} \leq \|f \|_{\ap} \quad \text{for all } |\lambda| \leq 1 \text{ and all } f \in \Lambda^p(\R^n),
\]
and  
\[
\lim_{\lambda \to 0} \|\lambda f \|_{\ap} = 0 \quad \text{for all } f \in \Lambda^p(\R^n);
\]
see \cite[Proposition~A.2]{alves2025F}.  
\par
Interestingly, the lack of homogeneity has deep consequences: the space $\Lambda^p(\R^n)$ is neither locally bounded nor locally convex \cite[Propositions~7.1 and 7.2]{alves2025F}, and its dual consists only of the zero functional \cite[Proposition~7.3]{alves2025F}. This highlights how fundamentally different $\Lambda^p(\R^n)$ is from the standard $L^p(\R^n)$. Nevertheless, many classical results have analogs in this setting. In~\cite{alves2025F}, versions of the dominated convergence and Vitali convergence theorems were established for $\Lambda^p(\R^n)$.
\par
Moreover, it follows from the definitions that if $f \in \Lambda^p(\R^n)$, then there exists a sequence $\{f_k\}_{k \in \N} \subseteq L^p(\R^n)$ that $\ap$-converges to $f$. Hence, $L^p(\R^n)$ is dense in $\Lambda^p(\R^n)$, and since $L^p(\R^n)$ is separable, so is $\Lambda^p(\R^n)$.
\par
Furthermore, when the underlying measure space has finite measure, for instance a bounded set $E \subseteq \R^n$, then $\Lambda^p(E)$ coincides with the space of all real-valued measurable functions on $E$; see \cite[Theorem~1.1]{alves2025F}. Thus, $\Lambda^p(\R^n)$ extends the space of measurable functions, equipped with the topology of convergence in measure, to the unbounded domain $\R^n$. In this sense, the asymptotic $L_p$ spaces retain features from both the standard $L^p(\R^n)$ and the $\F$-space of measurable functions. 

\section{Main result} \label{section_main_result}

In this section, we state the main result of the paper --- a characterization of the relatively compact subsets of $\Lambda^p(\R^n)$. Since $\Lambda^p(\R^n)$ is complete, relative compactness is equivalent to total boundedness.

\begin{theorem}[Kolmogorov--Riesz compactness theorem in $\Lambda^p(\R^n)$] \label{main_thm} 
A subset $\mathcal{F} \subseteq \Lambda^p(\R^n)$, $1 \leq p < \infty$, is totally bounded with respect to the $\F$-norm $\| \idot \|_{\ap}$ if and only if the following three conditions hold:
\begin{enumerate}[(i)]
\item For each $\varepsilon > 0$, there exists $R > 0$ such that 
\[
\int_{|x|>R} \min(|f|,1)^p \, \dx < \varepsilon^p
\]
for all $f \in \mathcal{F}$.
\smallskip
\item For each $\varepsilon > 0$, there exists $r > 0$ such that 
\[
\int_{\R^n} \min(|\tau_y f - f|,1)^p \, \dx < \varepsilon^p
\]
for every $y \in \R^n$ with $|y| < r$ and all $f \in \mathcal{F}$, where $\tau_y f(x) = f(x+y)$.
\smallskip
\item For each $\varepsilon > 0$, there exists $M > 0$ such that 
\[
\big| \{ |f| > M \} \big| < \varepsilon
\]
for all $f \in \mathcal{F}$.
\end{enumerate}
\end{theorem}

We note that the conditions in Theorem~\ref{main_thm} are sharp, in the sense that none of them can be deduced from the others. The first three examples in Section~\ref{section_examples} illustrate this. 
\par

Moreover, condition~(iii) is equivalent to almost equiboundedness, that is: for each $\varepsilon > 0$, there exists $M > 0$ such that for every $f \in \mathcal{F}$ there is a measurable set $S_f \subseteq \R^n$ with $|S_f| < \varepsilon$ and
\[
|f| \leq M \quad \text{on } S_f^c.
\]
Indeed, if condition~(iii) holds, it suffices to take
\[
S_f = \{|f| > M\}.
\]
Conversely, if the displayed condition holds, then
\[
\{|f| > M\} \subseteq S_f,
\]
and therefore
\[
\big| \{|f| > M\} \big| \le |S_f| < \varepsilon,
\]
which is exactly condition~(iii).

\section{Proof of Theorem \ref{main_thm}: Necessity} \label{section_proof_necessity}
In this section, we prove that a totally bounded family in $\Lambda^p(\R^n)$ satisfies the three conditions of Theorem~\ref{main_thm}. The proof is divided into three lemmas, each corresponding to one of the conditions.

\begin{lemma} \label{lemma1}
If a subset $\mathcal{F} \subseteq \Lambda^p(\R^n)$ is totally bounded (with respect to $\| \idot \|_{\ap}$) then for each $\eps>0$, there exists $R>0$ such that 
\begin{equation*} 
\int_{|x|>R} \min(|f|,1)^p \, \dx < \eps^p
\end{equation*} 
for all $f \in \mathcal{F}$.
\end{lemma}
\begin{proof}
Let $\mathcal{F} \subseteq \Lambda^p(\R^n)$ be totally bounded and $\eps > 0$. There exist $f_1, \ldots, f_m \in \Lambda^p(\R^n)$ such that \[\mathcal{F} \subseteq \bigcup_{i=1}^m B_{\ap}(f_i, \eps/2^{1+1/p}).\]
Since $f_i \in \Lambda^p(\R^n)$, there exists a measurable set $E_i$ with $|E_i| < \eps^p/(4m)$ so that $f_i \chi_{E_i^c} \in L^p(\R^n)$. This implies that there is $R_i>0$ such that 
\[\int_{E_i^c \cap \{|x| > R_i \}} |f_i|^p \, \dx < \frac{\eps^p}{2^{p+1}}. \]
Set $E = \bigcup_{i=1}^m E_i$ and $R = \max\{R_i \, : \, i = 1, \ldots, m\} $. Clearly  $|E| < \eps^p/4$. \par 
Let $f \in \mathcal{F}$. Then, for some $f_i$ we have $\|f_i - f \|_{\ap}< \eps/2^{1+1/p}$; in particular
\[\int_{|f-f_i|\leq1} |f_i-f|^p \, \dx < \frac{\eps^p}{2^{p+1}}, \qquad  \big|\{|f_i-f| > 1 \}\big| < \frac{\eps^p}{2^{p+1}} \leq \frac{\eps^p}{4}.\]
Set $G = E \cup \{|f_i-f| > 1 \}$. Then $|G| < \eps^p/2$ and 
\begin{align*}
\left(\int_{G^c \cap \{|x| > R \}} |f|^p  \, \dx \right)^{\frac{1}{p}} & \leq \left(\int_{G^c \cap \{|x| > R \}} |f_i-f|^p  \, \dx \right)^{\frac{1}{p}} + \left( \int_{G^c \cap \{|x| > R \}} |f_i|^p  \, \dx\right)^{\frac{1}{p}} \\
& \leq \left(\int_{|f_i-f| \leq 1} |f_i-f|^p  \, \dx \right)^{\frac{1}{p}} + \left( \int_{E_i^c \cap \{|x| > R_i \}} |f_i|^p  \, \dx\right)^{\frac{1}{p}} \\
& < \frac{\varepsilon}{2^{1/p}}.
\end{align*}
Consequently,
\begin{align*}
\int_{|x|>R} \min(|f|,1)^p \, \dx & = \int_{G \cap \{|x|>R \}} \min(|f|,1)^p \, \dx + \int_{G^c \cap \{|x|>R \}} \min(|f|,1)^p \, \dx \\
& \leq |G| + \int_{G^c \cap \{|x|>R \}} |f|^p \, \dx \\
& < \varepsilon^p
\end{align*}
which finishes the proof.
\end{proof}

\begin{lemma} \label{lemma2}
If a subset $\mathcal{F} \subseteq \Lambda^p(\R^n)$ is totally bounded (with respect to $\| \idot \|_{\ap}$) then for each $\eps>0$, there exists $r>0$ such that 
\begin{equation*}
\int_{\R^n} \min(|\tau_y f-f|,1)^p \, \dx < \eps^p
\end{equation*}
for every $y \in \R^n$ with $|y| < r$ and all $f \in \mathcal{F}$.
\end{lemma}
\begin{proof}
Assume that $\mathcal{F} \subseteq \Lambda^p(\R^n)$ is totally bounded and let $\varepsilon > 0$. There exist $f_1, \ldots, f_m \in \Lambda^p(\R^n)$ such that \[\mathcal{F} \subseteq \bigcup_{i=1}^m B_{\ap}(f_i,\eps/3).\]
For each $i = 1, \ldots, m$ there exists $\varphi_i \in C_c^\infty(\R^n)$ with $\|\varphi_i - f_i \|_{\ap}< \eps/9$ (see \cite[Proposition~6.3]{alves2025F}). Hence
\begin{align*}
\|\tau_y f_i - f_i  \|_{\ap} \leq \|\tau_y f_i - \tau_y \varphi_i  \|_{\ap} + \|\tau_y\varphi_i - \varphi_i  \|_{\ap} + \|\varphi_i - f_i  \|_{\ap}.
\end{align*}
The first term on the right-hand side equals the last one, and both are bounded by $\varepsilon/9$. The middle term is bounded by $\|\tau_y\varphi_i - \varphi_i  \|_p$. The smoothness of $\varphi_i$ guarantees the existence of a constant $r_i > 0$ so that $\|\tau_y\varphi_i - \varphi_i  \|_p < \eps/9$ whenever $|y| < r_i$. It follows that for $|y| < r_i$ we have 
\[\|\tau_y f_i - f_i  \|_{\ap} < \eps/3. \]
Set $r = \min\{r_i \, : \, i=1, \ldots, m \}$. Then, for $f \in \mathcal{F}$ and $y \in \R^n$ with $|y| < r$,
\begin{align*}
\|\tau_y f - f  \|_{\ap} \leq \|\tau_y f - \tau_y f_i  \|_{\ap} + \|\tau_y f_i - f_i \|_{\ap} + \|f_i - f  \|_{\ap}
\end{align*}
where $f_i$ is such that $f \in B_{\ap}(f_i,\eps/3)$. Each one of the terms on the right-hand side is bounded from above by $\eps/3$, yielding the desired conclusion.
\end{proof}

\begin{lemma} \label{lemma3}
If a subset $\mathcal{F} \subseteq \Lambda^p(\R^n)$ is totally bounded (with respect to $\| \idot \|_{\ap}$) then for each $\eps>0$, there exists $M>0$ such that 
\begin{equation*} 
\big| \{ |f| > M \} \big| < \eps
\end{equation*}
for all $f \in \mathcal{F}$.
\end{lemma}
\begin{proof}
Suppose, towards a contradiction, that there exists $\eps > 0$ such that for every $M>0$ one can find $f_M \in \mathcal{F}$ satisfying
\[\big| \{|f_M| > M \} \big| \geq \eps. \]  Since $\mathcal{F}$ is totally bounded, there are $f_1, \ldots, f_m \in \Lambda^p(\R^n)$ such that 
\[\mathcal{F} \subseteq \bigcup_{i=1}^m B_{\ap}(f_i,(\eps/4)^{1/p}).\]
For each $i=1, \ldots, m$, there exists a measurable set $E_i$ with $|E_i| < \eps/4$ such that $f_i \chi_{E_i^c} \in L^p(\R^n)$. Therefore
\begin{align*}
\big| E_i^c \cap \{|f_i| > M \} \big| & = \big| \{|f_i \chi_{E_i^c}| > M \} \big| \\
& \leq \frac{1}{M^p} \int_{E_i^c} |f_i|^p \, \dx \to 0 \quad \text{as} \ M \to \infty.
\end{align*}
Choose $M_i>0$ so that
 \[\big| E_i^c \cap \{|f_i| > M_i \} \big| < \frac{\eps}{4}. \] 
 We thus have for every $i = 1, \ldots, m$ that
 \begin{align*}
 \big|\{|f_i| > M_i \} \big| \leq | E_i| + \big| E_i^c \cap \{|f_i| > M_i \} \big| < \frac{\eps}{4} + \frac{\eps}{4} = \frac{\eps}{2}.
 \end{align*}
 Set $M_0 = \max\{M_i \, : \, i=1, \ldots, m \}$. There exists $f \in \mathcal{F}$ with $\big|\{|f| > M_0 + 1 \} \big| \geq \eps$. Moreover, $ \|f_j - f \|_{\ap}^p < \eps/4$ for some $j \in \{1, \ldots, m \}$. Define the sets 
 \[G = \{|f| > M_0 + 1 \} \quad \text{and} \quad H = \{ |f_j - f| > 1\} \]
and note that
 \[|H| \leq \|f_j - f \|_{\ap}^p < \frac{\eps}{4}. \]
 Now,
 \begin{align*}
 \eps \leq |G| \leq |H| + |G \cap H^c| < \frac{\eps}{4} + |G \cap H^c|
 \end{align*}
 whence
 \[|G \cap H^c| > \frac{3\eps}{4}. \]
 On $G \cap H^c$ we have 
 \[M_0 + 1 < |f| \leq |f_j - f| + |f_j| \leq 1 + |f_j| \]
 and hence 
 \[G \cap H^c \subseteq \{|f_j| > M_0 \} \]
 which implies 
 \[\frac{3\eps}{4} < |G \cap H^c| \leq \big| \{|f_j| > M_0 \} \big|. \]
However, since $M_0 \geq M_j$,
\[\big| \{|f_j| > M_0 \} \big| \leq \big| \{|f_j| > M_j \}\big| < \frac{\eps}{2} < \frac{3\eps}{4} \] 
which is a contradiction. The result follows.
\end{proof}

\section{Proof of Theorem \ref{main_thm}: Sufficiency} \label{section_proof_sufficiency}
Assume that $\mathcal{F} \subseteq \Lambda^p(\R^n)$ satisfies conditions (i), (ii) and~(iii) of Theorem~\ref{main_thm} and let $\eta > 0$ be given. According to condition~(iii) we can choose $M > 1$ such that 
\[\big| \{|f| > M \} \big| < \left(\frac{\eta}{2}\right)^p \]
for all $f \in \mathcal{F}$. Let $T_M$ be the truncation function defined for $t \in \R$ by $T_M(t) = \max \{ -M, \min\{t,M \} \}$, and define for each $f \in \mathcal{F}$ its truncated version 
\[f_M(x) = T_M(f(x)) = \begin{cases}
M, \quad & \text{if} \ f(x) > M, \\
f(x), \quad & \text{if} \ |f(x)| \leq M,\\
-M, \quad & \text{if} \ f(x) < -M. \\
\end{cases} \]
Since $f_M \in \Lambda^p(\R^n) \cap L^\infty(\R^n)$ it follows that $f_M \in L^p(\R^n)$. \par 
Now, given $f \in \mathcal{F}$ let $G =\{|f| > M \}$. Note that 
\begin{equation*}
|f(x) - f_M(x)| = \begin{cases}
|f(x)| - M, \quad & \text{if} \ x \in G, \\
0, \quad & \text{if} \ x \in G^c.
\end{cases}
\end{equation*}
Therefore
\begin{align*}
\int_{\R^n} \min(|f-f_M|,1)^p \, \dx = \int_G \min(|f-f_M|,1)^p \, \dx \leq |G| < \left(\frac{\eta}{2}\right)^p
\end{align*}
and hence, for every $f \in \mathcal{F}$,
\[\|f - f_M \|_{\ap} < \frac{\eta}{2}. \]
\par 
We proceed to show that the family $\mathcal{F}_M = \{ f_M \, : \, f \in \mathcal{F}\}$ satisfies conditions~(i) and (ii) of Theorem \ref{main_thm}. Regarding the first condition we note that since $M> 1$ then \[\min(|f_M|,1) = \min(|f|,1) \] being clear that $\mathcal{F}_M$ satisfies condition~(i). Regarding condition~(ii) we note that the truncation map $T_M$ is Lipschitz continuous with constant $1$, which implies that 
\[|f_M(x+y) - f_M(x)| = |T_M(f(x+y)) - T_M(f(x))| \leq |f(x+y) - f(x)|. \]
Therefore
\[\min(|f_M(x+y) - f_M(x)|,1) \leq \min(|f(x+y) - f(x)|,1) \]
and so
\[ \int_{\R^n} \min(|f_M(x+y) - f_M(x)|,1)^p \, \dx \leq \int_{\R^n} \min(|f(x+y) - f(x)|,1)^p \, \dx.\]
\par 
The next step is to prove that $\mathcal{F}_M$ is totally bounded in $L^p(\R^n)$ (with respect to $\|\idot \|_p$). From condition~(i) of Theorem~\ref{main_thm} we have  
\[\int_{|x| > R} \min(|f_M|,1)^p \, \dx < \left( \frac{\eps}{M} \right)^p \]
for some $R > 0$ and all $f_M \in \mathcal{F}_M$. Then 
\[\int_{|x| > R} |f_M|^p \, \dx \leq M^p \int_{|x| > R} \min(|f_M|,1)^p \, \dx < \eps^p \]
and thus $\mathcal{F}_M$ satisfies condition~(i) of Theorem~\ref{thm_KR}. \\
Moreover, given $\eps>0$, choose $r > 0$ from condition~(ii) of Theorem~\ref{main_thm} so that if $|y| < r$ then 
\[\int_{\R^n} \min(|f_M(x+y) - f_M(x)|,1)^p \, \dx < \frac{\eps^p}{1+(2M)^p}. \]
We have
\begin{align*}
\int_{\R^n} |f_M(x+y) - f_M(x)|^p \, \dx \leq & \ \int_{|\tau_y f_M - f_M| \leq 1} |f_M(x+y) - f_M(x)|^p \, \dx \\
& + (2M)^p \int_{|\tau_y f_M - f_M| > 1} 1 \, \dx \\
\leq & \ (1+(2M)^p) \int_{\R^n} \min(|f_M(x+y) - f_M(x)|,1)^p \, \dx \\
< & \ \eps^p.
\end{align*}
It follows by Theorem \ref{thm_KR} that $\mathcal{F}_M$ is totally bounded in $L^p(\R^n)$. Let $h_1, \ldots, h_m \in L^p(\R^n)$ be such that 
\[ \mathcal{F}_M \subseteq \bigcup_{i=1}^m B_p(h_i,\eta/2). \]
Then, given $f \in \mathcal{F}$, there exists $i \in \{ 1, \ldots, m\}$ so that 
\begin{align*}
\|f - h_i \|_{\ap} & \leq \|f - f_M \|_{\ap} + \|f_M - h_i \|_{\ap} \\
& \leq \eta/2 + \|f_M - h_i \|_p \\
& < \eta
\end{align*}
which completes the proof.

\section{Examples} \label{section_examples}
The first three examples of this section concern sequences of functions that satisfy only two of the conditions of Theorem~\ref{main_thm} but violate a third. This shows that Theorem~\ref{main_thm} is sharp in the sense that none of the conditions is redundant.
\begin{example}
Let $f_k(x) = k^{1/p} \chi_{[0,1]}(x)$. The sequence $\{f_k \}_{k\in \N}$ satisfies conditions~(i) and (ii) but violates condition~(iii). \par 
The first condition is clear since for $R > 1$ we have
\[\int_{|x| > R} \min(|f_k|,1)^p \, \dx = 0 \]
for all $k \in \N$. \par 
Regarding the second condition, given $\eps>0$ we choose $r = \min\{1,\eps^p/2\}$ so that for $|y| < r$ and $k\in \N$ it holds
\[\int_\R \min(|f_k(x+y) - f_k(x)|,1)^p \, \dx = 2|y| < 2r \leq \eps^p. \]
Now, we note that given $M>0$, for $k > M^p$ we have 
\[\big| \{|f_k| > M \} \big| = 1\]
and hence the third condition is not satisfied. 
\par Therefore, this sequence is neither totally bounded in $\Lambda^p(\R)$ nor in $L^p(\R)$ (in fact, it is not even bounded in $L^p(\R)$).
\end{example}

\begin{example}
Let $g_k(x)= \chi_{[k,k+1]}(x)$. The sequence $\{g_k \}_{k\in \N}$ satisfies conditions~(ii) and~(iii) but violates condition~(i). \par 
We start with condition~(iii). Simply note that $|g_k| \leq 1$ and hence for $M > 1$ we have 
\[\big| \{|g_k| > M \} \big| = 0\]
for all $k \in \N$. \par 
Condition~(ii) is analogous to the previous example: given $\eps > 0$, take $r = \min\{1,\eps^p/2\}$ so that for $|y| < r$ and all $k \in \N$,
\[\int_\R \min(|g_k(x+y) - g_k(x)|,1)^p \, \dx = 2|y| < 2r \leq \eps^p. \] \par 
Regarding the failure of condition~(i), we observe that given $R > 0$, choosing $k \geq R$ yields
\[\int_{|x| > R} \min(|g_k|,1)^p \, \dx = 1. \] \par 
It follows that $\{g_k \}_{k\in \N}$ is not totally bounded in $\Lambda^p(\R)$ (nor in $L^p(\R)$).
\end{example}
\begin{example}
Let $h_k(x) = r_k(x) \chi_{[0,1]}(x)$ where $r_k$ is the $k^{\text{th}}$ Rademacher function on $[0,1]$ (and extended by zero elsewhere), that is, $r_k(x) = \operatorname{sign}(\sin(2^k \pi x))$. The sequence $\{h_k \}_{k\in \N}$ satisfies conditions~(i) and (iii) but violates condition~(ii). \par 
The first and third conditions are clear (take $R > 1$ for (i) and $M > 1$ for (iii)), so we focus on the failure of condition~(ii). Let $r > 0$ and choose $k \in \N$ so that $y = 2^{-k} < r$. Note that for $x \in [0,1-2^{-k}]$ we have $r_k(x+y) = - r_k(x)$. Thus
\begin{align*}
\int_{\R} \min(|h_k(x+y) - h_k(x)|,1)^p \, \dx & \geq\int_0^{1-2^{-k}} \min(|2r_k(x)|,1)^p \, \dx \\
& = 1-2^{-k} \\
& \geq \frac{1}{2}
\end{align*}
which proves that $\{h_k \}_{k\in \N}$ does not satisfy condition~(ii).
\end{example}

The last two examples concern sequences that are totally bounded in~$\Lambda^p(\R^n)$ but not in~$L^p(\R^n)$.

\begin{example}
Let $1 \leq p < \infty$ and consider for each $k \in \N$ the function $u_k(x) = k^{1/p} \chi_{[k,k+1/k]}(x).$ The sequence $\{u_k\}_{k\in \N}$ asymptotically $L_p$-converges to the zero function, and hence it is totally bounded in~$\Lambda^p(\R)$. However, it is not totally bounded in $L^p(\R)$. Indeed, we have
\begin{align*}
\int_\R \min(|u_k(x)|,1)^p \, \dx & = \int_{k}^{k+1/k} \min(k^{1/p},1)^p \, \dx  \leq \frac{1}{k}  \to 0 \quad \text{as} \ k \to \infty,
\end{align*}
which proves that $\{u_k\}_{k\in \N}$ $\ap$-converges to~$0$, but given $R > 0$, for $k \geq R$ it holds
\begin{align*}
\int_{|x| > R} |u_k(x)|^p \, \dx \geq \int_R^\infty k \chi_{[k,k+1/k]}(x) \, \dx = 1,
\end{align*}
therefore $\{u_k\}_{k\in \N}$ does not satisfy the first condition of Theorem~\ref{thm_KR}, and hence it is not totally bounded in~$L^p(\R)$.  
\end{example}

\begin{example}
Let $1 < p < \infty$ and consider for each $k \in \N$, $v_k(x) = x^{-1} \chi_{[1/k, \infty)}$. Then, $\{v_k\}_{k \in \N}$ is not bounded in $L^p(\R)$, and hence it is not totally bounded in $L^p(\R)$, yet it is totally bounded in $\Lambda^p(\R)$ since it $\ap$-converges to $v \in \Lambda^p(\R)$ given by $v(x) = x^{-1} \chi_{(0,\infty)}(x).$
\end{example}

\section*{Acknowledgments}
This research was supported by the Austrian Science Fund (FWF), project 10.55776/F65.

\end{document}